\documentclass{article}

\usepackage{amsmath}
\usepackage{amsthm}
\usepackage{amssymb}
\usepackage[left=2.5cm,right=2.5cm,top=3cm,bottom=3cm]{geometry}
\usepackage{amscd}
\usepackage{stmaryrd}
\usepackage[normalem]{ulem}
\usepackage{MnSymbol}
\usepackage{bbm}

\usepackage[utf8]{inputenc}
\usepackage[cyr]{aeguill}
\usepackage{xspace}

\usepackage[pdftex]{graphicx}
\usepackage{enumerate}

\usepackage{array}

\usepackage[small,nohug,heads=vee]{diagrams}
\diagramstyle[labelstyle=\scriptstyle]

\theoremstyle{thm} \newtheorem{thm}{Theorem}
\newtheorem{prop}{Proposition} [section]
\newtheorem{lem}[prop]{Lemma}
\newtheorem{corol}[prop]{Corollary}

\theoremstyle{definition} 
\theoremstyle{remark} 
\theoremstyle{remark} \newtheorem{rem}[prop]{Remark}
\theoremstyle{definition} \newtheorem{defi}[prop]{Definition}

\usepackage{hyperref}
\usepackage{varioref}
\usepackage[hyperpageref]{backref}

\usepackage{xcolor}
\hypersetup{
    colorlinks,
    linkcolor={red!60!black},
    citecolor={blue!60!black},
    urlcolor={blue!50!black}
}

\newcommand{\quotient}[2]{{\left.\raisebox{.2em}{$#1$}\middle/\raisebox{-.2em}{$#2$}\right.}}

\newcommand{\lcm} {\mathrm{lcm}}

\title{Jet differentials on toroidal compactifications of ball quotients}
\author{Beno\^it Cadorel}
\date{}

\let\OLDthebibliography\thebibliography
\renewcommand\thebibliography[1]{
  \OLDthebibliography{#1}
  \setlength{\parskip}{0pt}
  \setlength{\itemsep}{0.13em plus 0.3ex}
}

\begin{document}

\maketitle

\begin{abstract} We give explicit estimates for the volume of the Green-Griffiths jet differentials of any order on a toroidal compactification of a ball quotient. To this end, we first determine the growth of the logarithmic Green-Griffiths jet differentials on these objects, using a natural deformation of the logarithmic jet space of a given order, to a suitable weighted projective bundle. Then, we estimate the growth of the vanishing conditions that a logarithmic jet differential must satisfy over the boundary to be a standard one.
\end{abstract}

\section{Introduction}

When dealing with complex hyperbolicity problems, finding global jet differentials on manifolds is an important question, since they permit to give restrictions on the geometry of the entire curves. Let us recall a few basic facts concerning Green-Griffiths jet differentials, which can be found in \cite{dem12a}. Let $X$ be a complex projective manifold. Then, for any $k$, there exists a (singular) variety $X_k^{GG} \overset{\pi_k}{\longrightarrow} X$, and an orbifold line bundle $\mathcal O_{X_k^{GG}}(1)$ on it, such that for any $m$, $E_{k, m}^{GG} \Omega_X = (\pi_k)_{\ast} \mathcal O_k(m)$ is a vector bundle whose sections are holomorphic differential equations of order $k$, and degree $m$, for some suitable notion of weighted degree.
 
In \cite{dem11}, Demailly proves that if $V$ is a complex projective manifold of general type, then for any $k$ large enough, the Green-Griffiths jet differentials of order $k$ will have maximal growth, or equivalently, $\mathcal O_{X_k^{GG}}(1)$ is big. Finding an effective $k$ for which this property holds is an interesting question, whose answer depends on the context: in \cite{dem11}, Demailly uses his metric techniques to give an effective lower bound on $k$ in the case of hypersurfaces of $\mathbb P^n$.

We propose here a method to obtain a similar effective result in the case of \emph{toroidal compactifications of ball quotients} (see \cite{mok12} for the main properties of these manifolds). Specifically, we will find a combinatorial lower bound on the volume of $E_{k, \bullet}^{GG} \Omega_X$, \emph{valid for any $k$} :
\begin{thm} \label{thmvol} Let $\overline{X}$ be a toroidal compactification of a ball quotient by a lattice with only unipotent parabolic isometries. Then, for any $k \in \mathbb N$, we have the following lower bound on the volume of the $k$-th order Green-Griffiths jet differentials:
\begin{align} \nonumber 
\mathrm{vol} (E^{GG}_{k,\bullet} \Omega_{\overline{X}} ) \geq \frac{1}{(k!)^n} & \left[ \; \left( \frac{(K_{\overline{X}} + D)^n}{(n+1)^n} \sum_{\left\{u_1 \leq ... \leq u_n \right\} \subset S_{k,n}} \frac{1}{u_1 ... u_n} \right) \right. \label{minorationfinale} \\
&  \left. + (-D)^n \sum_{1 \leq i_1 \leq ... \leq i_n \leq k} \frac{1}{i_1 ... i_n} \right],
\end{align}
where $S_{k,n}$ is the ordered set 
$$
S_{k,n} = \left\{ 1_1< ...< 1_{n+1} < 2_1 < ... < 2_{n+1} < ...  < k_1< ...< k_{n+1} \right\}.
$$
\end{thm}

Here, the fractions $\frac{1}{u_1 ... u_n}$ are to be computed by forgetting the indexes on the integers in the set $S_{k,n}$. In \cite{cad16}, the particular case where $k=1$ (i.e. the case of \emph{symmetric differentials}) was already proved, under the additional assumption that $\Omega_{\overline{X}}$ is nef. Our formula removes this hypothesis, and extends the result to any order $k$.

Using the results of \cite{baktsi15}, it is not hard to derive explicit orders $k$ for $E_{k, \bullet}^{GG} \Omega_{\overline{X}}$ to have maximal growth:
\begin{corol} \label{corolorder} Let $\overline{X} = \overline{\quotient{\mathbb B^n}{\Gamma}}$ be a toroidal compactification of a ball quotient. Let $k \in \mathbb N$. Then, under any of the following hypotheses:
\begin{enumerate}
\item $n \in \left[|4,5\right|]$ et $k > e^{-\gamma} e^{ -(-D)^n ( (n-2) n! + 1)}$ ;
\item $n \geq 6$ et $k > e^{-\gamma} e^{\frac{\frac{\pi^2}{6} (n-2)n! + 1}{ \frac{n+1}{2 \pi} - 1}}$,
\end{enumerate}
the line bundle $\mathcal O_{\overline{X}^{GG}_k}(1)$ is big. For the first values of $n \geq 4$, this yields the lower bounds for $\log k$ displayed in Table \ref{fig:table}.
\begin{table}[h!]
\begin{center}
\setlength{\tabcolsep}{7pt}
\begin{tabular} {c|ccccc}
$n$ & $4$ & $5$ & $6$ & $7$ & $8$ \\
\hline
$\log k$ & $-\gamma +5\, (-D^4)$  & $-\gamma + 361\, (D^5)$  &  $41\,534$ & $151\,711$  & $920\,325$ 
\end{tabular}
\caption{Effective lower bounds on $\log k$ to have $\mathcal O_{\overline{X}^{GG}_k}(1)$ big} 
\label{fig:table}
\end{center}
\end{table}
\end{corol}

The starting point to prove Theorem \ref{thmvol} consists in giving a more algebraic interpretation of the central metric construction of \cite{dem11}. Let us give the main ideas about this construction. For any complex manifold $X$, the Green-Griffiths jet differential spaces $X_k^{GG}$ can be deformed into a weighted projective bundle, using the standard construction of the Rees algebra. More specifically, there exists a family $\mathcal X_k^{GG} \longrightarrow X \times \mathbb C$ such that for any $\lambda \in \mathbb C^\ast$, the specialization $(\mathcal X^{GG}_k)_{\lambda} \longrightarrow X \times \left\{ \lambda \right\}$ is isomorphic to $X_k^{GG}$, and the specialization $(\mathcal X_k^{GG})_0 \longrightarrow X \times \left\{ 0 \right\}$ is isomorphic to the \emph{weighted projective bundle} ${\rm P} \left(T_X^{(1)} \oplus ... \oplus T_X^{(k)} \right)$. This last bundle is defined to be the quotient of $T_X \oplus ... \oplus T_X \longrightarrow X$ by the $\mathbb C^\ast$-action $\lambda \cdot (v_1, ..., v_k) = (\lambda v_1, ..., \lambda^k v_k)$. Moreover, there is a natural orbifold line bundle $\mathcal O_{\mathcal X_k^{GG}}(1)$ on the family $\mathcal X_k^{GG}$ whose restriction to the fibers $\left( \mathcal X_k^{GG} \right)_\lambda$ gives the tautological bundles of $X_k^{GG}$ and ${\rm P}(T_X^{(1)} \oplus ... \oplus T_X^{(k)})$. The metric used in \cite{dem11} can actually be seen as a singular metric on $\mathcal O_{\mathcal X_k^{GG}}(1)$; it is constructed in such a way that its specialization to the zero fiber ${\rm P} (T_X^{(1)} \oplus ... \oplus T_X^{(k)})$ is induced by some metric on $T_X$.

One convenient feature about this family $\mathcal X_k^{GG}$ is the fact that it permits to interpret the intersection products on the jet spaces $X_k^{GG}$ in terms of the intersection theory on ${\rm P}(T_X^{(1)} \oplus ... T_X^{(k)})$. When dealing with Chow groups computations, these last spaces share many properties with the usual weightless projective spaces. In the first part of our work, we will recall some results about the intersection theory with rational coefficients for a weighted projective spaces ${\rm P}(E_1^{(a_1)} \oplus ... \oplus E_p^{(a_p)})$, which were proved first by Al-Amrani \cite{alamrani97}. Since we work with rational coefficients instead of integer ones, the study is somewhat simplified; for the reader's convenience, we will explain how we could prove these results by following \cite{fulton98} in a standard way. 

The other reason why studying the family $\mathcal X_k^{GG}$ is interesting is the fact that the main positivity properties (e.g. nefness, ampleness) of the tautological line bundle $\mathcal O(1) \longrightarrow {\rm P}\left(T_X^{(1)} \oplus ... \oplus T_X^{(k)} \right)$ can be extended from the fiber over $0$ to other fibers over $\lambda \in \mathbb C^\ast$, i.e. to the line bundle $\mathcal O_{X_k^{GG}} (1)$. Moreover, the positivity properties of $\mathcal O(1)$ on ${\rm P}(T_X^{(1)} \oplus ... \oplus T_X^{(k)})$ are directly related to the ones of the vector bundle $T_X$. More generally, we will show in Section \ref{positivitysect} that if $E_1^\ast, ..., E_p^\ast$ are ample (resp. nef), then the orbifold line bundle $\mathcal O(1) \longrightarrow {\rm P}(E_1^{(a_1)} \oplus ... \oplus E_p^{(a_p)})$ is ample (resp. nef) in the orbifold sense for any choice of weights $a_1, ..., a_p$. This will imply in particular the following result:

\begin{prop} Let $X$ be a complex projective manifold. Assume that $\Omega_X$ is ample (resp. nef). Then for any $k \in \mathbb N^\ast$, $\mathcal O_{X_k^{GG}}(1)$ is ample (resp. nef) in the orbifold sense.
\end{prop} 

Getting back to the case of a ball quotient, we will use the logarithmic version of the previous discussion, and Riemann-Roch theorem in the orbifold case (see \cite{toen99}) to obtain an estimate for the volume of the Green-Griffiths logarithmic jets differentials $E_{k, \bullet}^{GG} \Omega_{\overline{X}}(\log D)$ in terms of the Segre class of the weighted direct sum $T_{\overline{X}} (- \log D)^{(1)} \oplus ... \oplus T_{\overline{X}} (- \log D)^{(k)}$. This last Segre class can be in turn expressed in terms of the standard Segre class $s_\bullet (T_{\overline{X}} (- \log D))$. An application of Hirzebruch proportionality principle in the non-compact case (see \cite{mum77}) will give our final estimate on $\mathrm{vol} (E_{k, \bullet}^{GG} \Omega_{\overline{X}} ( \log D))$, which will be the first member of the estimate \eqref{minorationfinale}.

Finally, it will remain to relate the growth of the logarithmic jet differentials to the growth of the standard ones. To to this, we will simply bound from above the sections of the coherent sheaves $\mathcal Q_{k,m}$, defined for any $k$ and $m$ by the exact sequence
$$
0 \longrightarrow E_{k, m}^{GG} \Omega_{\overline{X}} \longrightarrow E_{k, m}^{GG} \Omega_{\overline{X}} (\log D) \longrightarrow \mathcal Q_{k,m} \longrightarrow 0.
$$
We will find a suitable filtration on the sheaves $\mathcal Q_{k,m}$, in such a way that the graded terms are locally free above the boundary $D$, and can be expressed in terms of the vector bundles $\Omega_D$ and $N_{D/\overline{X}}$. Then, using Riemann-Roch computations and the fact that $D$ is a disjoint union of abelian varieties, we will be able to bound $h^0(\mathcal Q_{k,m})$ from above, for a fixed $k$, and $m$ going to $+ \infty$. This will give the second term in the estimate \eqref{minorationfinale}.

\subsection*{Acknowledgments} During the preparation of this work, the author was partially supported by the the French ANR project ”FOLIAGE”, ProjectID: ANR-16-CE40-0008. \\I would like to thank my advisor Erwan Rousseau for his guidance and his support, and Julien Grivaux for many helpful and enlightening discussions. I thank also the anonymous referee for his suggestions which I hope have permitted to improve the quality and clarity of this article.

\section{Segre classes of weighted projective bundles} \label{chernsect}

We will now recall some results, first proved by Al-Amrani \cite{alamrani97}, permitting to construct Chern classes of weighted projective bundles. We will state the results in the simpler setting of Chow rings with rational coefficients.  

\begin{defi} \label{defiprojbundle} Let $X$ be a complex algebraic projective variety. Consider a family $(E_i, a_i)_{1 \leq i \leq p}$, where the $E_i$ are vector bundles on $X$, and the $a_i$ are positive integers. The \emph{weighted projectivized bundle}  associated with the datum $(E_i, a_i)$ is the projectivized scheme of the graded $\mathcal O_X$-algebra $\mathrm{Sym} (E_1^{(a_1)} \oplus ... \oplus E_p^{(a_p)})^\ast$, defined as
$$
\mathrm{Sym} (E_1^{(a_1)} \oplus ... \oplus E_p^{(a_p)})^\ast = \mathrm{Sym}\; E_1^\ast {}^{(a_1)} \otimes_{\mathcal O_X} ... \otimes_{\mathcal O_X} \mathrm{Sym} \; E_p^\ast {}^{(a_p)},
$$
where, for any $i$, $\mathrm{Sym} E_i^\ast {}^{(a_i)}$ is the graded $\mathcal O_X$-algebra generated by sections of $E_i^\ast {}^{(a_i)}$ in degree $a_i$. We will denote this scheme by ${{\rm P}}(E_1^{(a_1)} \oplus ... \oplus E_p^{(a_p)})$;  remark that we use here the geometric convention for projectivized bundles.

We will say, by abuse of language, that $E_1^{(a_1)} \oplus ... \oplus E_p^{(a_p)}$ is a \emph{weighted direct sum}, or even a \emph{weighted vector bundle}.
\end{defi}

\begin{prop} The variety ${\rm P}( E_1^{(a_1)} \oplus ... \oplus E_p^{(a_p)})$ has a natural orbifold structure (or a structure of Deligne-Mumford stack), for which the tautological line bundle $\mathcal O(1)$ is naturally defined as an orbifold line bundle. Moreover, this orbifold line bundle is \emph{locally ample}, in the sense that the local isotropy groups of the orbifold structure act transitively on the fibres of $\mathcal O(1)$ (see for example \cite{rosstho09}). Besides, if $\lcm(a_1, ..., a_p) | m$, the bundle $\mathcal O(m)$ can be identified to a standard line bundle on ${\rm P}(E_1^{(a_1)} \oplus ... \oplus E_r^{(a_p)})$.
\end{prop}

\begin{proof} We can naturally endow ${\rm P}(E_1^{(a_1)} \oplus ... \oplus E_r^{(a_p)})$ with a structure of Artin stack $\mathcal P$, since it can be considered as a quotient stack
$$
\quotient{ E_1 \oplus ... \oplus E_p }{\mathbb C^\ast},
$$
where $\mathbb C^\ast$ acts by $\lambda \cdot (v_1, ..., v_p) = (\lambda^{a_1} v_1, ..., \lambda^{a_p} v_p)$. 

Locally on $X$, the weighted projectivized bundle can be trivialized as a product of the base with a weighted projectivized space $\mathbb P(a_1, ..., a_p)$, where each $a_i$ appears $\mathrm{rk}\; E_i$ times. Consequently, the Artin stack $\mathcal P$ has locally an orbifold structure, which makes it an orbifold stack. The claims on $\mathcal O(1)$ are local, and they can be proved directly using \cite{dol82} and \cite{rosstho09}.
\end{proof}

Let us start our review of the properties of the Chow groups with rational coefficients of the weighted projectivized bundles.

\begin{prop} Let $E_1^{(a_1)} \oplus ... \oplus E_p^{(a_p)}$ be a weighted direct sum. Let us denote the natural projection by $p : {\rm P}_X \left(E_1^{(a_1)} \oplus ... \oplus E_p^{(a_p)} \right) \longrightarrow X$. For any $k$, there is an isomorphism
\begin{equation} \label{isomfond}
A_k \, {\rm P}_X \left(E_1^{(a_1)} \oplus ... \oplus E_p^{(a_p)} \right) _{\mathbb Q} \cong \bigoplus_{j = 0}^r \left( A_{k - r +j} X \right)_{\mathbb Q},
\end{equation} where $r = \sum_{j =1}^p \mathrm{rk}\, E_j - 1$.
\end{prop}

To prove this result, we can start by checking it in the case where $X$ is an affine scheme. In that case, the weighted projective bundle is a quotient of a standard (trivial) projective bundle by a finite group, and it suffices to use the fact that such a quotient induces an isomorphism on the Chow rings with rational coefficients. We can then use the localization exact sequence to prove the general result.

Using the isomorphism \eqref{isomfond}, we can now define the Segre classes associated with a weighted direct sum $E_1^{(a_1)} \oplus ... \oplus E_p^{(a_p)}$.

\begin{defi} \label{defisegre} Let $X$ be a projective algebraic variety of dimension $n$, and let $E_1^{(a_1)} \oplus ... \oplus E_p^{(a_p)}$ be a weighted direct sum on $X$. Let $q : {\rm P}(E_1^{(a_1)} \oplus ... \oplus E_p^{(a_p)}) \longrightarrow X$ be the natural projection. 

If $k \in \left[|0, n \right|]$, the $k$-th Segre class of $E_1 ^{(a_1)} \oplus ... \oplus E_p^{(a_p)}$ is defined as an endomorphism of $(A_\ast X)_{\mathbb Q}$. If $\alpha \in (A_l X)_{\mathbb Q}$, let

$$
s_k \left( E_1^{(a_1)} \oplus ... \oplus E_p^{(a_p)} \right) \cap \alpha = \frac{1}{m^{k + r}} q_\ast \left( c_1 \mathcal O(m)^{r + k} \cap q^\ast \alpha \right).  
$$
where $r = \sum_{i} \mathrm{rk} E_i - 1$, and $m = \mathrm{lcm}(a_1, ..., a_p)$.
\end{defi}

\begin{rem} \label{remdivisible} In Definition \ref{defisegre}, we could have replaced $m$ by any integer divisible by $\mathrm{lcm} (a_1, ..., a_p)$. The important fact used here is that $\mathcal O(m)$ is a standard line bundle, which allows us to define its first Chern class in the usual way.
\end{rem}

There is a Whitney formula for the weighted projective bundles, which permits to express the Segre classes $s_j \left( E_1^{(a_1)} \oplus ... \oplus E_p^{(a_p)} \right)$ in terms of the $s_\bullet(E_j)$ and of the weights $(a_j)$:

\begin{prop} \label{propformulesegre}
Let $E_1^{(a_1)} \oplus  ... \oplus E_p^{(a_p)}$ be a weighted projective sum. We have
\begin{equation} \label{whitneypoids}
s_\bullet \left( E_1^{(a_1)} \oplus ... \oplus E_p^{(a_p)} \right) = \frac{\mathrm{gcd}(a_1, ..., a_p)}{a_1 \; ... \; a_p} \prod_{1 \leq j \leq p} s_\bullet \left( E_j^{(a_j)}\right),
\end{equation}
where, for any vector bundle $E$ and any weight $a \in \mathbb N$, we have $s_\bullet\left(E^{(a)}\right) = \frac{1}{a^{\mathrm{rk E} - 1}}\sum_{j \geq 0} \frac{s_j(E)}{a^j}$.
\end{prop}

To prove this result, we can use the "splitting principle" to get back to the case where the $E_i$ are all line bundles $L_i$. Now, denote $P = {\rm P}(L_0^{(a_0)} \oplus ... \oplus L_r^{(a_r)})$, and $p : P \longrightarrow X$ the canonical projection. Then, for some $m \in \mathbb N$, there exists a section of $(p^\ast L_0)^{\otimes l_0} \otimes \mathcal O_P(m)$ cutting out the subvariety ${\rm P}(L_1^{(a_1)} \oplus ... \oplus L_r^{(a_r)})$ with some computable multiplicity. As in \cite{fulton98}, we can use this fact to relate the Segre class $s_\bullet (L_1^{(a_1)} \oplus ... \oplus L_r^{(a_r)})$ with the classes $s_\bullet (L_0^{(a_0)} \oplus ... \oplus L_r^{(a_r)})$ and $c_\bullet (L_0)$. The formula then follows by induction.

\section{Positivity of weighted vector bundles} \label{positivitysect}

We now study the extension of the usual positivity properties of vector bundles to the case of weighted vector bundles.

\begin{defi} Let $\mathbb E = E_1^{(a_1)} \oplus ... \oplus E_p^{(a_p)}$ be a weighted direct sum. We say that $\mathbb E^\ast = E_1^\ast {}^{(a_1)} \oplus ... \oplus E_p^\ast {}^{(a_p)}$ is \emph{ample} (resp. \emph{nef}) if for any $m \in \mathbb N$ divisible enough, the (standard) line bundle  $\mathcal O(m)$ is ample (resp. nef) on ${\rm P}(\mathbb E)$.
\end{defi}

\begin{rem} With the terminology of \cite{rosstho09}, saying that $\mathbb E^\ast$ is ample amounts to saying that $\mathcal O(1)$ is orbi-ample on ${\rm P}_X(\mathbb E)$, the tautological orbifold line bundle being locally ample by \cite{dol82}.
\end{rem}

We will see that the positivity properties of weighted vector bundles are exactly similar to the ones of the usual vector bundles, and can be proved in the same manner, following \cite{lazpos2}.

\begin{prop} \label{ampleprop} Assume that $E_1^\ast$, ..., $E_p^\ast$ are ample on $X$. Then,
\begin{enumerate}
\item For any coherent sheaf $\mathcal F$ on $X$, there exists $m_1 \in \mathbb N$ such that, for any $m \geq m_1$, the sheaf 
$$
\mathcal F \otimes \left( \bigoplus_{a_1 l_1 + ... + a_p l_p = m} S^{l_1} E_1^\ast \otimes ... S^{l_p} E_p^\ast \right)
$$
is globally generated.
\item For any ample divisor $H$ on $X$, there exists $m_2 \in  \mathbb N$ such that for any $m \geq m_2$, the sheaf
$$
\bigoplus_{a_1 l_1 + ... + a_p l_p = m} S^{l_1} E_1^{\ast} \otimes ... \otimes S^{l_p} E_p^{\ast}
$$
is a quotient of a direct sum of copies of $\mathcal O_X(H)$. 
\item If $\mathrm{lcm}(a_1, ..., a_p) | m$, then $\mathcal O(m)$ is ample on ${\rm P}(E_1^{(a_1)} \oplus ... \oplus E_p^{(a_p)})$. In particular, $\mathcal O(1)$ is orbi-ample.
\end{enumerate}
\end{prop}
\begin{proof}1. This is easy to prove by induction on $p$ using the similar characterization of ample vector bundles.

2. It suffices to apply the point 1. to the sheaf $\mathcal F =\mathcal O(-H)$. 

3. Because of 2., there exist $m, N \in \mathbb N$ and a surjective morphism
$$
\mathcal O_{X} (H)^{\oplus N} \longrightarrow \bigoplus_{a_1 l_1 + ... a_p l_p = m} S^{l_1} E_1^\ast \otimes ... \otimes S^{l_p} E_p^\ast.
$$

Besides, because of Lemma \ref{lemsurj}, increasing $m$ if necessary, we can suppose that for any $q \in \mathbb N$, the following natural morphism of vector bundles on $X$ is surjective:
$$
S^{q} \left[ \bigoplus_{a_1 l_1 + ... a_p l_p = m} S^{l_1} E_1^\ast \otimes ... \otimes S^{l_p} E_r^\ast \right] \longrightarrow \bigoplus_{a_1 l_1 + ... + a_p l_p = mq} S^{l_1} E_1^\ast \oplus ... \oplus S^{l_p} E_p^\ast. 
$$

We obtain a surjective morphism of graded $\mathcal O_X$-algebras
$$
\bigoplus_{q \geq 0} S^q \left( \mathcal O_X(H)^{\oplus N} \right) \longrightarrow \bigoplus_{q \geq 0} \left[ \bigoplus_{a_1 l_1 + ... + a_p l_p = mq} S^{l_1} E_1^\ast \otimes ... \otimes S^{l_p} E_p^\ast \right].
$$
which determines an embedding ${\rm P}\left(E_1^{(a_1)} \oplus ... \oplus E_p^{(a_p)} \right) \hookrightarrow {\rm P}\left( \mathcal O_X(-H)^{\oplus N}\right)$, with 
$$
\mathcal O_{{\rm P}(\mathcal O_X(-H)^{\oplus N})}(1) |_{{\rm P}(E_1^{(a_1)} \oplus ... \oplus E_p^{(a_p)})} \cong \mathcal O(qm).
$$
Since the tautological line bundle on ${\rm P}(\mathcal O_X(-H)^{\oplus N})$ is ample (cf. \cite{lazpos2}), this ends the proof. 
\end{proof}

\begin{lem} \label{lemsurj}
Let $E_1, ..., E_p$ be $\mathbb C$-vector spaces, and let $a_1, ..., a_p \in \mathbb N^\ast$. Then, for any $m \in \mathbb N$ divisible enough, the natural linear maps
$$
S^q \left[ \bigoplus_{a_1 l_1 + ... + a_p l_p = m} S^{a_1} E_1 \otimes ... \otimes S^{a_p} E_p \right] \longrightarrow \bigoplus_{a_1 l_1 + ... + a_p l_p = mq} S^{a_1} E_1 \otimes ... \otimes S^{a_p} E_p
$$
are onto for all $q \geq 1$.
\end{lem}
\begin{proof}
Because of \cite{dol82}, if $m$ is sufficiently large and divisible by all $a_1, ..., a_r$, the (standard) line bundle  $\mathcal O(m)$ on the weighted projective space ${\rm P}_{\mathrm{pt}}(E_1^\ast {}^{(a_1)} \oplus ... \oplus E_p^\ast {}^{(a_p)} )$ is very ample. Consequently, there exists an integer $q \in \mathbb N$ such that $S^p H^0( \mathcal O(mq) ) \longrightarrow H^0(\mathcal O(mqp))$ is onto for all $p \geq 1$, which gives the result.
\end{proof}

We will now study the case of nef line bundles. We will prove the following result.
 
\begin{prop} \label{propnef}
Let $E_1^{(a_1)} \oplus ... \oplus E_p^{(a_p)}$ be a weighted direct sum. Assume that $E_1^\ast, ..., E_p^\ast$ are nef. Then, if $m$ is sufficiently divisible, the line bundle $\mathcal O(m)$ is nef on ${\rm P}\left(E_1^{(a_1)} \oplus ... \oplus E_p^{(a_p)}\right)$.
\end{prop}

To this aim, we will use the formalism of vector bundles twisted by rational classes (see \cite{lazpos2} for the definition and the positivity properties of these objects). As in the weightless case, we naturally define the notion of ampleness for a weighted sum of twisted vector bundles:

\begin{defi} \label{defiampletwist} We say that a weighted direct sum of twisted vector bundles of the form
$$
E_1<a_1 \delta>^{(a_1)} \oplus ... \oplus E_p<a_p \delta>^{(a_p)}
$$ is \emph{ample}, if for any $m$, divisible by $\mathrm{lcm}(a_1, ..., a_p)$ the $\mathbb Q$-line bundle $\mathcal O(m) \otimes \pi^\ast \mathcal O_X\left(m \delta \right)$ is ample on ${\rm P} \left(E_1^\ast {}^{(a_1)} \oplus ... E_p^\ast {}^{(a_p)} \right)$. 
\end{defi}

\begin{rem} \label{remtwists} In Definition \ref{defiampletwist}, we consider twists of the form $a_1 \delta, ..., a_r \delta$ with $\delta \in N^1 (X)_{\mathbb Q}$. This is related to the fact that if $E_1, ..., E_r$ are vector bundles, and if $L$ is a line bundle, we have, for any $m$,
\begin{align*}
\bigoplus_{a_1 l_1 + ... + a_r l_r = m} S^{l_1} (E_1^\ast \otimes L^{\otimes a_1}) \otimes ... & \otimes (E_r^\ast \otimes L^{\otimes a_r}) = \\
&  L^{\otimes m} \otimes \bigoplus_{a_1 l_1 + ... + a_r l_r = m} S^{l_1} E_1^\ast \otimes ... \otimes E_r^\ast,
\end{align*}
which implies in particular that the weighted projective bundle $P' = \\ {\rm P}\left((E_1 \otimes L^{\ast} {}^{\otimes a_1}) ^{(a_1)} \otimes ... \otimes (E_p \otimes L^{\ast} {}^{\otimes a_p})^{(a_p)}\right)$ is identified to the variety $P = {\rm P}(E_1^{(a_1)} \oplus ... \oplus E_r^{(a_r)})$, with $\mathcal O_{P'}(m) \cong \mathcal O_{P}(m) \otimes p^\ast L^{\otimes m}$.
\end{rem}
 
\begin{lem} \label{lemamplepoids} Let $E_1<a_1 \delta>, ..., E_r<a_r \delta>$ be twisted vector bundles on $X$. Assume that each $E_i^\ast<-a_i \delta>$ is ample. Then the weighted direct sum
$$
E_1^\ast <- a_1 \delta>^{(a_1)} \oplus ... \oplus E_r^\ast <- a_r \delta>^{(a_r)}
$$ is ample. 
\end{lem}
\begin{proof}
We follow directly the proof presented in \cite{lazpos2}. Because of Bloch-Gieseker theorem about ramified covers (see \cite[Theorem 4.1.10]{lazpos2}), there exist a finite, surjective, flat morphism $f : Y \longrightarrow X$, where $Y$ is a variety, and a divisor $A$ such that $f^\ast \delta \equiv_{lin} A$. We have a fibered diagram
$$
\begin{diagram}
P' = {\rm P}_Y\left(f^\ast E_1^\ast {}^{(a_1)} \oplus ... \oplus f^\ast E_r^\ast {}^{(a_r)} \right) & \rTo^{g} & {\rm P}_X(E_1^\ast {}^{(a_1)} \oplus ... \oplus E_r^\ast {}^{(a_r)}) = P\\
 \dTo & & \dTo \\
Y & \rTo^f & X
\end{diagram}.
$$
Let $Q = {\rm P}_Y \left(\left(f^\ast E_1^\ast \otimes \mathcal O(a_1 A) \right) {}^{(a_1)} \oplus ... \oplus \left( f^\ast E_r^\ast \otimes \mathcal O(a_r A) \right) {}^{(a_r)} \right)$. Then, we have a canonical identification $Q \cong P'$, which leads to identifying the line bundle $\mathcal O_Q(m)$ with $\mathcal O_{P'}(m) \otimes \pi_Y^\ast\mathcal O_Y(m A)$, as mentioned in Remark \ref{remtwists}.

Besides, the $\mathbb Q$-line bundle
\begin{equation} \label{fibredroiteseq}
g^\ast \left( \mathcal O_P(m) \otimes \pi^\ast \mathcal O_X\left(m \delta \right) \right)
\end{equation} is canonically identified to $\mathcal O_{P'}(m) \otimes \pi_Y^\ast\mathcal O_Y( m A)$, thus to $\mathcal O_Q(m)$. However, since each $E_i^\ast<- a_i \delta>$ is ample, and since $f$ is finite, each vector bundle  $f^\ast E_i^\ast \otimes \mathcal O(-a_i A)$ is ample. Because of Proposition \ref{ampleprop}, the line bundle $\mathcal O_Q(m)$ is ample, so the line bundle \eqref{fibredroiteseq} is ample. But $g$ is finite and surjective, so $\mathcal O_P(m) \otimes \pi^\ast \mathcal O_X\left(m \delta \right)$ is ample on $P$, which gives the result.
\end{proof}

\begin{proof} [Proof of Proposition \ref{propnef}] 
It suffices to show that for any ample class $h \in N_1(X)_{\mathbb Q}$, the class $\mathcal O(m) \otimes \pi^\ast \mathcal O(m h)$ is ample. Let $h$ be such a class.
Then since each $E_i^\ast$ is nef, the twisted vector bundles $E_i^\ast < a_i h>$ are ample for any $i$. Consequently, by Lemma \ref{lemamplepoids} and Definition \ref{defiampletwist}, if $\mathrm{lcm} (a_1, ..., a_r) | m$, the line bundle $\mathcal O(m) \otimes \pi^\ast \mathcal O_X (m h)$ is ample on ${\rm P}(E_1^\ast {}^{(a_1)} \oplus ... \oplus E_r^\ast {}^{(a_r)})$. This gives the result.
\end{proof}

\subsection{An example of combinatorial application}

We present a simple example of application of the previous discussion, which will turn out to be useful in Section \ref{applicationcompactifications}, where we deal with jet bundles on a toroidal compactification of a quotient of the ball.

\begin{prop} \label{propcombmajorationapplication} Let $k,n \in \mathbb N$. We have the following asymptotic upper bound, as $r \longrightarrow + \infty$ :
\begin{align*}
\sum_{j_1 + 2 j_2 + ... + k j_k = r} \frac{(j_1 + ... + j_k)^n}{n!} \leq  \frac{1}{k!} & \left[ \sum_{1 \leq i_1 \leq ... \leq i_n \leq k} \frac{1}{i_1 \; ... \; i_n}\right] \frac{r^{n + k -1}}{(n+k-1)!}  \\
& + O(r^{n +k -2}).
\end{align*}
\end{prop}
\begin{proof}
Let $X$ be an abelian variety of dimension $n$, endowed with an ample line bundle $L$. Because of Proposition \ref{ampleprop}, the weighted direct sum $L^{(1)} \oplus ... \oplus L^{(k)}$ is ample on $X$. This means that the orbifold line bundle $\mathcal O(1)$ is orbi-ample on $ P = {\rm P}_X \left(L^\ast {}^{(1)} \oplus ... \oplus L^{\ast} {}^{(k)} \right)$. By orbifold asymptotic Riemann-Roch theorem (\cite{toen99}, see also \cite{rosstho09}), we have then, for any $m \in \mathbb N$,
$$
h^0_{orb}( {\rm P}, \mathcal O(m)) \leq \int_{P} c_1 \mathcal O(1)^{n+k-1}\frac{m^{n+k-1}}{(n+k-1)!}+ O(m^{n+k-1}).
$$
However, because of Definition \ref{defisegre} and Proposition \ref{propformulesegre}, $\int_{X} c_1 \mathcal O(1)^{n+k-1}$ can be computed as 
\begin{align*}
\int_{P} c_1 \mathcal O(1)^{n+k-1} & = \int_{X} s_{n} \left(L^\ast {}^{(1)} \oplus ... \oplus L^\ast {}^{(k)}\right) \\
		 		& = \frac{1}{k!} \int_X \left\{ \left(\sum_{i} H^i \right) ... \left( \sum_{i} \frac{H^i}{l^i} \right) ... \left( \sum_{i} \frac{H^i}{k^i} \right) \right\}_n,
\end{align*}
where $H = c_1(L)$. Expending the computation yields
\begin{align*}
\int_{P} c_1 \mathcal O(1)^{n+k-1} & = \frac{(L^n)}{k!} \left[ \sum_{l_1 + ... + l_k = n} \frac{1}{1^{l_1} \; ... \; k^{l_k} }\right] \\
				& = \frac{(L^n)}{k!} \left[ \sum_{1 \leq i_1 \leq ... \leq i_n \leq k} \frac{1}{i_1 \; ... \; i_k} \right].
\end{align*}
To obtain the result, it suffices to remark that we can identify the vector space $$H^0(X, \bigoplus_{j_1 + 2 j_2 + ... + k j_k = m} L^{\otimes (j_1 + ... +j_k)})$$ to a subspace of the orbifold global sections of $\mathcal O(m)$. Thus :
$$
h^0 (X, \bigoplus_{j_1 + 2 j_2 + ... + k j_k = m} L^{\otimes (j_1 + ... + j_k)} ) \leq h^0_{orb}( P, \mathcal O(m)).
$$ 
Besides, a direct application of Riemann-Roch-Hirzebruch theorem and Kodaira vanishing theorem on $X$ gives
$$
h^{0} (X, L^{\otimes (j_1 + ... + j_k)}) = \frac{(j_1 + ... + j_k)^n}{n!} (L^n)
$$
if $j_1 + ... j_k \neq 0$. Combining all these equations, we get the inequality.
\end{proof}

We can also get back the following classical result.  
\begin{prop} \label{propestimeepoint} Let $a_0, ..., a_n \in \mathbb N^\ast$. Let $X = \mathbb P(a_0, ..., a_n)$ be the associated weighted projective space, endowed with its tautological orbifold line bundle $\mathcal O_X(1)$. We then have the asymptotic estimate
$$
h_{\mathrm{orb}}^0( X, \mathcal O_X(m)) = \frac{\mathrm{gcd}(a_0, ..., a_n)}{\prod_j a_j} \frac{m^n}{n!} + O(m^{n-1}).
$$
\end{prop}
\begin{proof}
It is clear that the weighted direct sum $\mathbb C^{(a_0)} \oplus ... \oplus \mathbb C^{(a_n)} \longrightarrow \mathrm{Spec} \;\mathbb C$ is ample, which means that  $\mathcal O_X(1)$ is orbi-ample. Then, using \cite{rosstho09} and Definition \ref{defisegre},
\begin{align*}
h_{\mathrm{orb}}^0( X, \mathcal O_X(m)) & = \int_{X} c_1 \mathcal O(1)^n \cdot \frac{m^n}{n!} + O(m^{n-1}) \\
	& = s_0( \mathbb C^{(a_0)} \oplus ... \oplus \mathbb C^{(a_n)}) \frac{m^n}{n!} + O(m^{n-1}).
\end{align*}
Besides, because of Proposition \ref{propformulesegre}, we have 
$$
s_0( \mathbb C^{(a_0)} \oplus ... \oplus \mathbb C^{(a_n)}) = \frac{\mathrm{gcd}(a_0, ..., a_n)}{\prod_j a_j},
$$ which gives the result.
\end{proof}

\section{Green-Griffiths jet bundles}

\subsection{Deformation of the jet spaces}

We first remark that for any projective complex manifold, there is a natural deformation of its Green-Griffiths jets spaces to a weighted projectivized bundles, which will permit us to apply the previous discussion to the study of jet differentials.

Let $X$ be a projective complex manifold. For $k \in \mathbb N$, we consider the Green-Griffiths jet differentials algebra $E_{k, \bullet}^{GG} \Omega_X$. Recall (cf. \cite{dem12a}) that $E_{k, \bullet}^{GG} \Omega_X$ is endowed with a natural filtration, which can be described as follows. 

For each $(n_1, ..., n_k) \in \mathbb N^k$, and any coordinate chart $U \subset X$, we define the $(n_1, ..., n_k)$-graded term as the following space of local jet differentials:
\begin{align*}
F^{(l_1, ..., l_k)} E_{k,m}^{GG} & \Omega_X (U) \\
& = \left\{  \left. \sum_{I = (I_1, ..., I_k)} a_I \; (f')^{I_1} ... (f^{(k)})^{I_k} \; 
 \right|  \; (|I_1|, ..., |I_k|) \leq (l_1, ..., l_k) \right\}.
\end{align*}
where for each $I_l = (p_1, ..., p_n)$, we write $(f^{(l)})^{I_l} = (f_1^{(l)})^{r_1} ... (f_n^{(l)})^{r_n}$. In the above formula, the lexicographic order on $\mathbb N^k$ is defined so that $(p_1, ..., p_k) < (n_1, ..., n_k)$ means that either $p_k < n_k$, or $p_k = n_k$ and $(p_1, ..., p_{k-1}) < (n_1, ..., n_{k-1})$ in the lexigraphic order for $\mathbb N^{k-1}$. 

The formula of derivatives of composed maps implies that these local definitions glue together to give a well defined $\mathbb N^k$-filtration $F^\bullet E_{k, m}^{GG} \Omega_X$, compatible with the $\mathcal O_X$-algebra structure on $E_{k, \bullet}^{GG} \Omega_X$, and which is increasing with respect to the lexicographic order. Moreover, the graded terms occur only for $(l_1, ..., l_k)$ such that $l_1 + 2 l_2 + ... + k l_k = m$, and we have, in this case:
$$
\mathrm{Gr}_F^{(l_1, ..., l_k)} \left( E_{k, m}^{GG} \right) = \mathrm{Sym}^{l_1} \Omega_X \otimes ... \otimes \mathrm{Sym}^{l_k} \, \Omega_X.
$$
By Definition \ref{defiprojbundle}, this means exactly that, as an $\mathcal O_X$-algebra,
\begin{equation} \label{eqgrad}
\mathrm{Gr}_F \left( E_{k, \bullet}^{GG} \right) \cong \mathrm{Sym}\, \Omega^{(1)}_X \otimes ... \otimes \mathrm{Sym} \, \Omega^{(k)}_X.
\end{equation}

We can use the Rees deformation construction (see for example \cite{BG96}), to construct a $\mathcal O_{X \times \mathbb C}$-algebra $\mathcal E^{GG}_{k, \bullet}$ on $X \times \mathbb C$, such that  for any $\lambda \in \mathbb C^\ast$, $\left. \mathcal E^{GG}_{k, \bullet} \right|_{X \times \left\{ \lambda \right\}}$ is identified to $E^{GG}_{k,\bullet} \Omega_{X}$, and $\left. \mathcal E^{GG}_{k, \bullet} \right|_{X \times \left\{ 0 \right\} }$ is identified to $\mathrm{Gr}_F \left( E_{k, \bullet}^{GG} \right)$. 
\medskip

Let us give a few details about the construction. We first define a graded $\mathcal O_{X \times \mathbb C^k}$-algebra $\widetilde{\mathcal E}^{GG}_{k, \bullet}$, as follows:
$$
\widetilde{\mathcal E}^{GG}_{k, m} (U \times \mathbb C^k) = \left\{ \left. \sum_{l_1, ..., l_k} u_{l_1, ..., l_k} t_1^{l_1} ... t_k^{l_k} \; \right| \; u_{l_1, ...,l_k} \in F^{(l_1, ..., l_k)} E_{k, m}^{GG} \Omega_X(U) \right\}
$$
where $(t_1, ..., t_k)$ are coordinates on $\mathbb C^k$. It is easy to check that for any $(\lambda_1, ..., \lambda_k) \in \mathbb C^k$ such that each $\lambda_j \neq 0$, we have $\widetilde{\mathcal E}^{GG}_{k, m}|_{X \times \left\{ (\lambda_1, ..., \lambda_k) \right\} } \cong E^{GG}_{k, m} \Omega_X$, and that $\widetilde{\mathcal E}_{k, m}^{GG}|_{X \times \{(0, ..., 0)\}} \cong \mathrm{Gr}_F (E_{k, m}^{GG} \Omega_X)$. Now, define $\mathcal E_{k, \bullet}^{GG}$ to be the pullback of $\widetilde{\mathcal E}^{GG}_{k, \bullet}$ by the embedding $(x, t) \in X \times \mathbb C \mapsto (x, t, ..., t) \in X \times \mathbb C^{k}$.

\begin{rem} While $\widetilde{\mathcal E}_{k, \bullet}^{GG}$ seems to be the natural object arising in the construction above, it is more tractable to work over $X \times \mathbb C$ with the sheaf $\mathcal E_{k, \bullet}^{GG}$. To define the latter, we could have used any embedding $t \in \mathbb C \longmapsto (t \alpha_1, ..., t \alpha_k) \in \mathbb C^k$, with $\alpha_i \neq 0$, so our choice $(\alpha_1, ..., \alpha_k) = (1, ..., 1)$ is rather arbitrary. The same phenomenon occurs in \cite{dem11}, where the metric on $\mathcal O_{X_k^{GG}}(m)$ constructed by Demailly depends on some auxiliary parameters $\epsilon_1, ..., \epsilon_k \in \mathbb R_+^\ast$.
\end{rem}

Applying the $\mathbf{Proj}$ functor, we obtain the following result.

\begin{prop} \label{propdeformation} For any complex projective manifold $X$, and for any $k \in \mathbb N^\ast$, there exists a morphism of varieties $\mathcal X^{GG}_k \longrightarrow X \times \mathbb C$, and an orbifold line bundle $\mathcal O_{\mathcal X^{GG}_k}(1)$ on $\mathcal X^{GG}_k$, such that :
\begin{enumerate}
\item for any $\lambda \in \mathbb C^\ast$, $\left. \mathcal X^{GG}_k \right|_{\lambda}$ is identified to $X^{GG}_k$, and $\left.\mathcal O_{\mathcal X^{GG}_k}(1) \right|_{\lambda}$ identified to $\mathcal O_{X^{GG}_k}(1)$ ;
\item the fibre $\left. \mathcal X^{GG}_k \right|_0$ is identified to the variety $\mathbf{Proj}_X \left( \mathrm{Gr}_F \left(E_{k, \bullet}^{GG} \right) \right) \cong {\rm P}_X \left( T_X^{(1)} \oplus ... \oplus T_X^{(k)} \right)$, and $\left.\mathcal O_{\mathcal X^{GG}_k}(1) \right|_0$ is identified to the tautological line bundle of this weighted projective bundle.
\end{enumerate}
\end{prop}
\begin{rem}
By construction, the identifications mentioned above already occur at the level of sheaves of algebras. To obtain the identifications when taking $\mathbf{Proj}$ functors, we just need to check that the gradings on these sheaves of algebras are compatible under these identifications. 
\end{rem}
\begin{proof}
Let us prove the second point, the first one being similar. By construction, we have a natural identification between 
$$
\mathcal E_{k, \bullet}^{GG}|_{X \times \{ 0 \}} = \bigoplus_{m \geq 0} \mathcal E_{k,m }^{GG}|_{X \times \{ 0 \}}
$$
and 
$$
\mathrm{Gr}_F (E_{k, \bullet}^{GG} \Omega_X) = \bigoplus_{m \geq 0} \mathrm{Gr}_F (E_{k, m}^{GG} \Omega_X)
$$
Moreover, this identification is compatible with the grading in $m$. Besides, by \eqref{eqgrad}, the latter sheaf is identified, as a sheaf of \emph{graded} algebras, with
$$
\mathrm{Sym} \, \Omega_X^{(1)} \otimes ... \otimes \mathrm{Sym} \, \Omega_X^{(k)} = \bigoplus_{m \geq 0} \left( \bigoplus_{l_1 + 2 l_2 + ... + k l_k = m} \mathrm{Sym}^{l_1} \Omega_X \otimes ... \otimes \mathrm{Sym}^{l_k} \Omega_X \right).
$$ 
Now, by Definition \ref{defiprojbundle}, the projectivized bundle associated to the latter sheaf of graded algebras, with respect to the grading in $m$, is ${\rm P}(T_X^{(1)} \oplus ... \oplus T_X^{(k)})$. This implies immediately the identifications of varieties and orbifold line bundles mentioned in the second point. 
\end{proof}

We can now show that some usual positivity properties of the cotangent bundle can be transmitted to the higher order jet differentials. 

\begin{prop} \label{propdiffpositivite} If $\Omega_X$ is ample (resp. nef), then for any $k \in \mathbb N^\ast$, $E^{GG}_{k, \bullet} \Omega_X$ is ample (resp. nef), meaning that $\mathcal O_{X^{GG}_k} (1)$ is ample (resp. nef) as an orbifold line bundle.
\end{prop}
\begin{proof}
Let $\mathcal X^{GG}_k \longrightarrow X \times \mathbb C$ be the variety given by Proposition \ref{propdeformation}, endowed with its orbifold line bundle $\mathcal O_{\mathcal X^{GG}_k} (1)$.

Assume first that $\Omega_X$ is ample. Then, because of Proposition \ref{ampleprop}, a suitable power of the tautological  line bundle $\mathcal O(m)$ is ample on ${\rm P} (T_X^{(1)} \oplus ... \oplus T_X^{(k)} )$ if $m$ is sufficiently divisible. Because of Proposition \ref{propdeformation}, it means that $\mathcal O_{\mathcal X^{GG}_k}(1) |_0$ is ample. By semi-continuity of the ampleness property, for any $\lambda \in \mathbb C^\ast$ in a Zariski neighborhood of $0$, the orbifold line bundle $\mathcal O_{\mathcal X^{GG}_k}(1) |_{\lambda}$ is ample. Again because of Proposition \ref{propdeformation}, this means exactly that $E^{GG}_{k, \bullet} \Omega_X$ is ample.

The case where $\Omega_X$ is nef is dealt with in the same manner, using Proposition \ref{propnef}, and the fact that if $\mathcal O_{X^{GG}_k}(1) |_{0}$ is nef, then $\mathcal O_{X^{GG}_k} (1) |_{\lambda}$ is nef for any very general $\lambda \in \mathbb C$ (see \cite{lazpos1}).
\end{proof}

The previous discussion extends naturally to the case of logarithmic jet differentials. We then have the following proposition.

\begin{prop} \label{deformlog} Let $(X, D)$ be a smooth log-pair. For any $k \in \mathbb N^\ast$, there exists a morphism $\mathcal X^{GG, \log}_k \longrightarrow X \times \mathbb C$ and an orbifold line bundle $\mathcal O_{\mathcal X^{GG, \log}_k} (1)$ on $\mathcal X^{GG, \log}_k$ such that
\begin{enumerate}
\item for any $\lambda \in \mathbb C^\ast$, $\left. \mathcal X^{GG, \log}_k \right|_{\lambda}$ is identified to $X^{GG, \log}_k$, and $\left. \mathcal O_{\mathcal X^{GG, \log}_k} (1) \right|_{\lambda}$ identified to $\mathcal O_{X^{GG}_k}(1)$ ;
\item the fibre $\left. \mathcal X^{GG, \log}_k \right|_0$ is identified to 
$$
\mathbf{Proj}_X \mathrm{Gr}_F \left(E_{k, \bullet}^{GG} \right) \cong \mathbb{P}_X \left( T_X(-\log D)^{(1)} \oplus ... \oplus T_X(-\log D)^{(k)} \right),
$$
and $\left. \mathcal O_{\mathcal X^{GG, \log}_k} (1) \right|_0$ is identified to the tautological orbifold line bundle of this weighted projectivized bundle.
\end{enumerate}
\end{prop}
\begin{proof} As before, it suffices to use the fact that $E^{GG}_{k,m} \Omega_X(\log D)$ admits a filtration whose graded algebra is $\mathrm{Sym}\, \Omega_X(\log D)^{(1)} \otimes ... \otimes \mathrm{Sym} \, \Omega_X(\log D)^{(k)}$.
\end{proof}

In this setting, Proposition \ref{propdiffpositivite} extends naturally:
\begin{prop} \label{proplogpositicite} Let $(X, D)$ be a smooth log-pair. If $\Omega_X (\log D)$ is nef (resp. ample), then for any $k \in \mathbb N^\ast$, $E^{GG}_{k, \bullet} \Omega_X (\log D)$ is nef (resp. ample), meaning that $\mathcal O_{X^{GG, \log}_k} (1)$ is nef (resp. ample) as orbifold line bundle.
\end{prop}
\begin{rem}
Proposition \ref{proplogpositicite} is actually only relevant for the nef property. Indeed, except when $X$ is a curve, the bundle $\Omega_X (\log D)$ cannot be ample in general: if $D$ is smooth and $\dim X \geq 2$, we have the residue map
$$
\Omega_X (\log D) \longrightarrow \mathcal O_D \longrightarrow 0,
$$
and the restriction of this map to $D$ shows that $\Omega_X (\log D)|_{D}$ has a trivial quotient. This implies that $\Omega_X (\log D)$ is not ample.
\end{rem}

The following result, combining two theorems of Campana and P\u{a}un \cite{campau15}, and Demailly \cite{dem11}, shows that the bigness of the canonical orbifold line bundle $\mathcal O(1)$ on ${\rm P}(T_X^{(1)} \oplus ... \oplus T_X^{(k)})$ for $k$ large enough suffices to characterize the manifolds of general type.

\begin{prop} Let $X$ be a projective smooth manifold. The following assertions are equivalent.
\begin{enumerate}[(i)]
\item $X$ is of general type;
\item for large $k$, $E^{GG}_{k, \bullet} X$ is big, meaning that the usual line bundle  $\mathcal O(m) \longrightarrow X^{GG}_k$ is big for $m$ sufficiently divisible ;
\item for large $k$, the orbifold line bundle $\mathcal O(1) \longrightarrow {\rm P}_X (T_X^{(1)} \otimes ... \otimes T_X^{(k)})$ is big, i.e. the line bundle $\mathcal O(m)$ is big for $m$ sufficiently divisible.
\end{enumerate}
\end{prop}
\begin{proof}
$(i) \Rightarrow (ii)$. This is the main result of \cite{dem11}. 

$(ii) \Rightarrow (iii)$. Let $k \in \mathbb N^\ast$ large enough, and consider a sufficiently divisible $m \in \mathbb N^\ast$. Let $\mathcal X^{GG}_k \longrightarrow X \times \mathbb C$ be the variety given by Proposition \ref{propdeformation}, endowed with its tautological orbifold line bundle $\mathcal O_{\mathcal X^{GG}_k}(1)$. For any $\lambda \in \mathbb C^\ast$, $\mathcal O_{\mathcal X^{GG}_k} (m) |_{\lambda}$ is identified to $\mathcal O_{X^{GG}_k} (m)$, which is big. Consequently , there exists a constant $C > 0$, such that for any $\lambda \in \mathbb C^{\ast}$,
$$
h^{0} \left(\left. \mathcal X_k^{GG} \right|_{\lambda}, \mathcal O_{\mathcal X^{GG}_k}(m)|_{\lambda}\right) \geq C m^{n + nk -1}.
$$
Since $\mathcal O_{\mathcal X^{GG}_k}(m)$ is flat on the base $\mathbb C$, we deduce by semi-continuity that
$$
h^{0} \left(\left. \mathcal X_k^{GG} \right|_{0}, \mathcal O_{\mathcal X^{GG}_k}(m)|_{0}\right) \geq C m^{n + nk -1}.
$$
Besides, $X^{GG}_k|_{0}$ et $\mathcal O_{\mathcal X^{GG}_k}(1)|_{0}$ are identified with ${\rm P}_X(T_X^{(1)} \oplus ... \oplus T_X^{(k)})$ and to its tautological line bundle, so the previous inequality means exactly that $\mathcal O(1)$ is big on ${\rm P}(T_X^{(1)} \oplus ... \oplus T_X^{(k)})$.

$(iii) \Rightarrow (i)$. This result is proved in \cite{campau15}.
\end{proof}

\section{Application to the toroidal compactifications of ball quotients} \label{applicationcompactifications}

Let $\Gamma \in \mathrm{Aut}(\mathbb B^n)$ be a lattice with only unipotent parabolic isometries. Then, by \cite{AMRT} and \cite{mok12}, we can compactify the quotient $X = \quotient{\mathbb B^n}{\Gamma}$ into a \emph{toroidal compactification} $\overline{X} = X \sqcup D$, where $D$ is a disjoint union of abelian varieties. From now, on, $\overline{X}$ will always denote such a toroidal compactification of a ball quotient.

Let $k \in \mathbb N^\ast$, and let $\mathcal X^{GG, \log}_k \longrightarrow X \times \mathbb C$ be the family given by Proposition \ref{deformlog}. Denote by $P^{GG, \log}_k \subset \mathcal X^{GG, \log}_k$ the fibre above $0 \subset \mathbb C$, which is isomorphic to ${\rm P}_X(T_X(-\log D)^{(1)} \oplus ... \oplus T_X (-\log D)^{(k)})$. Let us also denote by $\mathcal O_P(1)$ the orbifold tautological line bundle on this weighted projective bundle.

\begin{prop} The orbifold line bundle $\mathcal O_{\overline{X}^{GG, \log}_k} (1) \longrightarrow \overline{X}^{GG, \log}_k$ is nef.
\end{prop}
\begin{proof}
The vector bundle $\Omega_{\overline{X}} (\log D)$ is nef because of \cite{cad16}. Thus, the result comes from Proposition \ref{propdiffpositivite}.
\end{proof}

If $m_0 = \lcm(1, ..., k)$, the standard line bundle $\mathcal O_{\overline{X}^{GG, \log}_k} (m_0)$ is nef. This gives the following asymptotic expansion:
\begin{align} \nonumber
h^0(\overline{X}, E^{GG}_{k, l m_0} \Omega_{\overline{X}} (\log D)) & = h^0(\overline{X}^{GG}_k, \mathcal O_{\overline{X}^{GG, \log}_k} (l m_0)) \\ \nonumber
& = \chi(\overline{X}^{GG}_k, \mathcal O_{\overline{X}^{GG, \log}_k} (l m_0)) + O(l^{n + nk - 2}) \\
		      & = \left( \int_{\overline{X}^{GG, \log}_k} c_1 \mathcal O(m_0)^{m + nk - 1} \right) l^{n + nk - 1} + O( l^{n + nk - 2} ). \label{dvtasymptotique}
\end{align}

By Proposition \ref{deformlog}, $\overline{X}_k^{GG, \log}$ and $P_k^{GG, \log}$ are members of the same flat family $\mathcal X^{GG, \log}_k \longrightarrow \mathbb C$. Thus, since the first Chern class is a topological invariant, we can compute the leading coefficient of this last expansion, as follows:
$$
\int_{\overline{X}^{GG, \log}_k} c_1 \mathcal O(m_0)^{m + nk - 1} = \int_{P_k^{GG, \log}} c_1 \mathcal O_{P}(m_0)^{m + nk - 1}.
$$
Then, using Definition \ref{defisegre}, we find
$$
\int_{\overline{X}^{GG, \log}_k} c_1 \mathcal O(m_0)^{m + nk - 1} = m_0^{n + nk -1} \int_{\overline{X}} s_{n}(T_{\overline{X}} (-\log D)^{(1)} \oplus ... \oplus T_{\overline{X}} (- \log D)^{(k)}).
$$

If we insert this equation in \eqref{dvtasymptotique}, we see that if $m \in \mathbb N$ is divisible by $m_0$, we have
\begin{align*}
h^0(\overline{X}, E^{GG}_{k, m} \Omega_{\overline{X}} & (\log D)) =  \\
& \int_{\overline{X}} s_{n}(T_{\overline{X}} (-\log D)^{(1)} \oplus ... \oplus T_{\overline{X}}  (- \log D)^{(k)}) \; \frac{m^{n +nk -1}}{(n + nk -1)!} \\
&  + O(m^{n+nk -2}).
\end{align*}

This gives the following value for the volume of $E^{GG}_{k, \bullet} \Omega_{\overline{X}} (\log D)$:
\begin{equation} \label{volumeexpr}
\mathrm{vol} \left(E^{GG}_{k, \bullet} \Omega_{\overline{X}} (\log D) \right) = \int_{\overline{X}} s_{n} (T_{\overline{X}} (- \log D)^{(1)} \oplus ... \oplus T_{\overline{X}} (- \log D)^{(k)} ).
\end{equation}

\subsection{Combinatorial expression of the volume. Uniform lower bound in $k$}

The volume \eqref{volumeexpr} can be expressed as a certain universal polynomial with rational coefficients in the Chern classes of $T_{\overline{X}}(-\log D)$. The same polynomial, applied to the Chern  classes of $T_{\mathbb P^n}$ over $\mathbb P^n$, permits to compute $\int_{\mathbb P^n} s_{n} (T_{\mathbb P^n}^{(1)} \oplus ... \oplus T_{\mathbb P^n}^{(k)})$, and Hirzebruch proportionality principle in the non-compact case (see \cite{mum77}) implies
\begin{align*}
\int_{\overline{X}} s_{n} (T_{\overline{X}} (- \log D)^{(1)} \oplus ...  \oplus T_{\overline{X}} &  (- \log D)^{(k)} ) = \\
& (-1)^n \frac{(K_{\overline{X}} + D)^n}{(n+1)^n} \int_{\mathbb P^n} s_{n} (T_{\mathbb P^n}^{(1)} \oplus ... \oplus T_{\mathbb P^n}^{(k)} )
\end{align*}

Using Proposition \ref{propformulesegre}, we can give an explicit combinatorial expression of this last quantity. Indeed, if we let $H = c_1 \mathcal O_{\mathbb P^n} (1)$, since $s_{\bullet} (T_{\mathbb P^n}) =  \left( \sum_{i= 1}^n (-1)^i H^i \right)^{n+1}$, we find
\begin{align*}
(-1)^n & \int_{\mathbb P^n}  s_{n} (T_{\mathbb P^n}^{(1)} \oplus ... \oplus T_{\mathbb P^n}^{(k)} )   \\
	& = \frac{(-1)^n}{(k!)^n} \left\{ \left( \sum_{i=1}^n (-1)^i H^i \right)^{n+1} \left( \sum_{i=1}^n (-1)^i \frac{H^i}{2^i} \right)^{n+1} \cdot .... \cdot \left( \sum_{i=1}^n (-1)^i \frac{H^i}{k^i} \right)^{n+1}  \cdot \left[ \mathbb P^n \right]\right\}_0 \\
& = \frac{1}{(k!)^n} \sum_{l_{1,1} + l_{1,2} + ... + l_{n+1, k} = n} \frac{1}{1^{l_{1,1} + l_{2, 1} + ... + l_{n+1, 1}} \cdot ... \cdot k^{ l_{1, k} + ... + l_{n+1, k}}} (H^n \cdot \left[ \mathbb P^n \right]).
\end{align*}
where each index $l_{i,j}$ ($i \in \left[| 1, n+1 \right|]$, $j \in \left[| 1, k \right|]$) represents a possible choice of power for $H$ in the $i$-th factor of the product $\left( \sum_{l = 1}^n (-1)^l \frac{H^l}{j^l} \right)^{n+1}$.
Thus, we see that choosing exponents $(l_{i,j})_{1 \leq i \leq n+1, 1 \leq j \leq k}$ such that $\sum_{i,j} l_{i,j} = n$ amounts to choosing a non-decreasing sequence $u_1 \leq ... \leq u_n$ of elements of the ordered set 
$$
S_{k, n} = \left\{1_1 < ...< 1_{n+1} < 2_1 < ...< 2_{n+1} < ... < k_{1} < ... < k_{n+1} \right\},
$$
where each integer between $1$ and $k$ is repeated $n+1$ times. 
The bijection between the set of choices of $(l_{i,j})$ and the set of sequences $u_1 \leq ... \leq u_n$ can easily be made explicit : to $(l_{i,j})$, we associate the sequence $(u_i)$, where the element $j_m$ is repeated $l_{m, j}$ times. Thus, we find
$$
(-1)^n \int_{\mathbb P^n} s_{n} (T_{\overline{X}}^{(1)} \oplus ... \oplus T_{\overline{X}}^{(k)} ) = \frac{1}{(k!)^n} \sum_{\left\{u_1 \leq ... \leq u_n \right\} \subset S_{k, n}} \frac{1}{u_1 ... u_n},
$$
where, in the quotient appearing on the right hand side, we compute the product by treating the elements of $S$ as ordinary integers (we forget their indexes).

We then find an explicit combinatorial formula for the volume of logarithmic jet differentials of order $k$ :
\begin{equation} \label{estimeevolcomb}
\mathrm{vol} \left( E_{k, \bullet}^{GG} \Omega_{\overline{X}} \left( \log D \right) \right) = \frac{\left(K_{\overline{X}} + D \right)^n}{(n+1)^n (k!)^n} \sum_{\left\{ u_1 \leq ... \leq u_n \right\} \subset S_{k, n} } \frac{1}{u_1 ... u_n}. 
\end{equation}

It is not hard to use this formula to obtain a more tractable lower bound on the volume. Indeed, we have:
\begin{align} \nonumber
n! \sum_{ \left\{u_1 \leq ... \leq u_n \right\} \subset S_{k,n} } \frac{1}{u_1 ... u_n} & \geq \sum_{ (u_1, ..., u_n) \in S_{k, n}^n } \frac{1}{u_1 ... u_n}  \\ \nonumber
	& = \left( \sum_{u \in S_{k, n}} \frac{1}{u} \right)^n \\
	& = (n+1)^n \left( 1 + 1/2 + ... + 1/k\right)^n \geq (n+1)^n (\log k + \gamma)^n \label{minorationestimeeouvert}
\end{align}
where, in the first inequality, we use the fact that for any ordered set $\left\{ u_1 \leq ... \leq u_n \right\}$, the number of distinct $n$-uples $(v_1, ..., v_n)$ having the same elements is at least $n!$. The letter $\gamma$ represents the Euler-Mascheroni constant.

We obtain the following lower bound, \emph{valid for any $k \geq 1$}: 
$$
\mathrm{vol} \left( E^{GG}_{k, \bullet} \Omega_{\overline{X}} (\log D) \right) \geq \left( K_{\overline{X}} + D \right)^n \frac{( \log k + \gamma)^n}{(k!)^n \, n!}.
$$
This formula can be seen as an effective version of the asymptotic estimates of \cite{dem11}, in the case of logarithmic jet differentials on a toroidal compactification of a ball quotient.

\section{Upper bound on the vanishing conditions on the boundary}

We will now study the number of vanishing conditions on the boundary that a logarithmic jet differential must satisfy to be a standard one.

For any $k \in \mathbb N^\ast$, and any $m \in \mathbb N$, we define a sheaf $\mathcal Q_{k, m}$, supported on $D$, in the following manner:
\begin{equation} \label{suiteexacte}
0 \longrightarrow E_{k,m}^{GG} \Omega_{\overline{X}} \longrightarrow E_{k,m}^{GG} \Omega_{\overline{X}} (\log D) \longrightarrow \mathcal Q_{k,m} \longrightarrow 0.
\end{equation}
Then, we have:
\begin{equation} \label{minorationstdlog}
h^0(\overline{X}, E_{k,m}^{GG} \Omega_{\overline{X}} ) \; \geq \; h^0(\overline{X}, E_{k,m}^{GG} \Omega_{\overline{X}} (\log D) ) - h^0(\overline{X}, \mathcal Q_{k,m}).
\end{equation} 

\subsection{Filtration on the quotient $\mathcal Q_{k,m}$}

Our goal is to obtain an upper bound on $h^0(\mathcal Q_{k,m})$, as  $m \longrightarrow +\infty$, with fixed $k \in \mathbb N$. To do this, we will produce a sufficiently sharp filtration on the sheaf $\mathcal Q_{k,m}$, so that the graded terms are locally free $\mathcal O_D$-modules. We will then the bound from above the number of global sections of these graded terms.

\begin{prop} The inclusion of \eqref{suiteexacte} preserves the natural filtrations on $E_{k,m}^{GG} \Omega_{\overline{X}}$ and $E_{k,m}^{GG} \Omega_{\overline{X}} (\log D)$.
\end{prop}
\begin{proof}
We only need to check this locally: this inclusion sends an jet differential equation of the form $\prod_{i,l} (f_i^{(l)})^{a_{i,l}}$ on $\prod_{i \neq n,l} (f_i^{(l)})^{a_{i,l}} \cdot \prod_{l} z_n^{a_{i,l}} (\frac{f_n^{(l)}}{z_n})^{a_{i,l}}$. The exponents of the different $f_i^{(l)}$ are then preserved by the inclusion, so the natural filtrations are also preserved. 
\end{proof}

Consequently, $\mathcal Q_{k,m}$ admits a induced filtration $F_1$, whose graded terms can be written as a quotient of the corresponding graded terms in $E_{k,m}^{GG} \Omega_{\overline{X}}$ and $E_{k,m}^{GG} \Omega_{\overline{X}} (\log D)$ :
\begin{align} \label{graduefil1} 
& \mathrm{Gr}^{F_1}_{\bullet}  \left(\mathcal Q_{k,m} \right) =  \\
	 \nonumber &  \bigoplus_{l_1 + 2 l_2 + ... + k l_k = m} \; \quotient{\left[ S^{l_1} \Omega_{\overline{X}} (\log D) \otimes ... \otimes S^{l_k} \Omega_{\overline{X}} (\log D)\right]}{\mathrm{Im} \left[ S^{l_1} \Omega_{\overline{X}} \otimes ... \otimes S^{l_k} \Omega_{\overline{X}} \right]}.
\end{align}

We will now produce successive refinements of the filtration $F_1$, until we obtain a filtration whose graded terms are all locally free $\mathcal O_{D}$-modules. We can already simplify the quotient appearing in \eqref{graduefil1}, using the following elementary result.

\begin{lem} Let $\mathcal E_{1}, ... \mathcal E_{l}$ be $\mathcal O_{\overline{X}}$-modules. For any $i$, we consider a sub-module $\mathcal E'_{i} \hookrightarrow \mathcal E_{i}$. Then the quotient $\quotient{\mathcal E_{1} \otimes ...\otimes  \mathcal E_{l}}{\mathrm{Im} \left( \mathcal E'_{1} \otimes ... \otimes \mathcal E'_{l} \right)}$ admits a filtration whose  $i$-th graded term can be identified with
$$
\mathcal E'_{1} \otimes ... \otimes \mathcal E'_{i-1} \otimes \left( \quotient{\mathcal E_{i}}{\mathcal E'_{i}} \right) \otimes \mathcal E_{i+1} \otimes ... \otimes \mathcal E_{l}.
$$
\end{lem}
\begin{proof}
It suffices to consider the filtration induced on the quotient sheaf $\quotient{\mathcal E_{1} \otimes ... \mathcal E_{l}}{\mathrm{Im} \left( \mathcal E'_{1} \otimes ... \otimes \mathcal E'_{l} \right) }$ by the images of any of the sheaves appearing in the sequence of morphisms
$$
\mathcal E'_1 \otimes ... \otimes \mathcal E'_{l} \longrightarrow ... \longrightarrow \mathcal E'_1 \otimes ... \otimes \mathcal E'_{i-1} \otimes \mathcal E_i \otimes ... \otimes \mathcal E_l \longrightarrow ... \longrightarrow \mathcal E_1 \otimes ... \otimes \mathcal E_l.
$$
\end{proof}

We deduce from this proposition and the previous one the existence of a filtration $F_2$ on $\mathcal Q_{k,m}$, whose graded module can be written
$$
\mathrm{Gr}_{\bullet}^{F_2} \left( \mathcal Q_{k,m} \right) = \bigoplus_{l_1 + 2 l_2 + ... + k l_k = m} \bigoplus_{i = 1}^{k} \; \; S^{l_1} \Omega_{\overline{X}} \otimes ... \otimes \mathcal S_{l_i} \otimes ... \otimes S^{l_k} \Omega_{\overline{X}} (\log D),
$$
where $\mathcal S_{l} = \quotient{S^{l} \Omega_{\overline{X}}(\log D) }{S^{l} \Omega_{\overline{X}}}$.

The $\mathcal O_{\overline{X}}$-modules $\mathcal S_l$ can be in turn filtered in $\mathcal O_D$-modules, using a filtration that was first introduced in \cite{campau07}. For completeness, we will describe this filtration in our special case.

\begin{prop} \label{propfiltS} For any $l \in \mathbb N$, $\mathcal S_{l}$ is endowed with a filtration, whose graded terms are $\mathcal O_D$-modules, written
$$
\mathrm{Gr}_{\bullet} (\mathcal S_l) = \bigoplus_{j = 0}^{l} \bigoplus_{s = 0}^{j} \left( N_{D/\overline{X}}^{\ast} \right)^{\otimes s} \otimes S^{l-j} \Omega_{D}.
$$
\end{prop}
\begin{proof}
According to \cite{mok12}, each boundary component $T_b$ admits a tubular neighborhood $U$, quotient of its universal cover $\widehat{U} \subset \mathbb C^{n-1} \times \mathbb C$ by a lattice $\Lambda \subset \mathbb C^{n-1}$. The component $T_b$ can be identified to the quotient of $\mathbb C^{n-1} \times \left\{ 0 \right\}$ by $\Lambda$. 
Let $D^\circ = \mathbb C^{n-1} \times \left\{ 0 \right\} \subset \widehat{U}$. The elements $a \in \Lambda$ act on $\Omega_{\widehat{U}} (\log D^0)$ in the following way:
$$
\left\{ \begin{array}{c}
	a \cdot \frac{d z_n}{z_n} = \frac{d z_n}{z_n} + \sum_{i=1}^{n-1} \gamma_i(a) dz_i ; \\
	a \cdot d z_i = d z_i \; \text{if} \; 1 \leq i \leq n - 1,
\end{array}
\right.
$$
where the $\gamma_i : \mathbb C^{n-1} \longrightarrow \mathbb C$ are $\mathbb R$-linear maps.
The natural filtration by the degree of $\frac{dz_n}{z_n}$ in $ S^l \Omega_{\widehat{U}} (\log D^\circ)$ is consequently preserved by $\Lambda$, and induces a filtration $G_l$ on $S^l \Omega_{U}(\log D))$ whose graded terms are globally trivial and can be written 
$$
\mathrm Gr^{G_l}_{j} (S^l \Omega_{U} (\log D)) = \left[ \left( \frac{dz_n}{z_n} \right)^j \right] \cdot S^{l - j} \Omega_D.
$$

This expression in local coordinates shows that the induced filtration by $G_l$ on $S^l \Omega_{U}$ admits as general graded term
$$
\mathrm{Gr}_j^{G_l \cap S_l \Omega_U} (S^l \Omega_U) = \mathcal I_{j D} \otimes_{\mathcal O_U} \left[ \left( \frac{dz_n}{z_n} \right)^j \right] \cdot S^{l - j} \Omega_D,
$$
where $\mathcal I_{jD}$ is the sheaf of ideals of the divisor $j D$. Consequently, $G_l$ induces a new filtration on the quotient $\quotient{S^l \Omega_{U}(\log D)}{S^l \Omega_{U}}$, whose graded terms are
$$
\mathrm{Gr}_{\bullet} (\mathcal S_l) = \mathcal O_{jD} \otimes_{\mathcal O_U} S^{l - j} \Omega_{D}.
$$

To obtain the proposition, it suffices to refine this last filtration, remarking that $\mathcal O_{jD} = \quotient{\mathcal O_{\overline{X}} }{ \mathcal I_{j D}}$ is itself filtered by 
$$
0 \subset \; \quotient{\mathcal I_{(j-1) D}}{\mathcal I_{jD}} \; \subset ... \subset \; \quotient{\mathcal I_{l D}}{\mathcal I_{jD}} \; \subset ... \subset  \; \mathcal O_{jD},
$$
whose successive quotients can be identified to $\quotient{ \mathcal I_{l D}}{\mathcal I_{(l+1) D}} \simeq \left( N_{D/ \overline{X}}^\ast \right)^{\otimes l}$.
\end{proof}

We can consequently refine the filtration $F_2$, to obtain a new one $F_3$, whose graded module is
\begin{align*}
\mathrm{Gr}^{F_3}_{\bullet} & (\mathcal Q_{k,m})   = \\ 
& \bigoplus_{i = 1}^k \; \bigoplus_{l_1 + 2 l_2 + ... + k l_k = m} \;  \bigoplus_{j_i = 0}^{l_i} \bigoplus_{s_i = 0}^{j_i} \; \;   
  \left( S^{l_1} \Omega_{\overline{X}} \otimes ... \otimes  S^{l_{i-1}} \Omega_{\overline{X}} \right. \\ 
& \otimes  \left[ \left( N^\ast_{D/\overline{X}} \right)^{\otimes s_i} \otimes S^{l_i - j_i} \Omega_{D} \right]  \left. \otimes S^{l_{i+1}} \Omega_{\overline{X}} (\log D) \otimes ... \otimes S^{l_k} \Omega_{\overline{X}} (\log D) \right).
\end{align*} 

Each one of the terms of this direct sum can be seen as an $\mathcal O_D$-module. Besides, we have seen in the proof of Proposition \ref{propfiltS} that $S^l \Omega_{\overline{X}} (\log D) |_{D}$ admits a natural filtration whose graded terms are trivial:
$$
\mathrm{Gr}_\bullet \left( S^l \Omega_{\overline{X}} (\log D) \right) = \bigoplus_{j=0}^l S^j \Omega_D.
$$
On the other hand, since each boundary component admit a tubular neighborhood, we have
$$
\Omega_{\overline{X}}|_D = N^\ast_{D/\overline{X}} \oplus \Omega_D,
$$
so $S^l \Omega_{\overline{X}} |_D \simeq \bigoplus_{j=0}^l \left( N_{D/\overline{X}}^\ast \right)^j \otimes S^{l-j} \Omega_D$. We can consequently refine a last time the filtration on $\mathcal Q_{k, m}$, to obtain the following proposition.

\begin{prop} For any $k, m \in \mathbb N^\ast$, there exists a filtration $F_{\bullet} \mathcal Q_{k,m}$, whose graded module is an $\mathcal O_D$-module written 
\begin{align} \label{termegradfinal} 
\mathrm{Gr}^F_{\bullet} & \left( \mathcal Q_{k,m} \right) = \\
 \nonumber &  \bigoplus_{i = 1}^k \; \bigoplus_{l_1 + 2 l_2 + ... + k l_k = m} \; \bigoplus_{j_1= 0}^{l_i} ... \bigoplus_{j_k = 0}^{l_k} \bigoplus_{s_i = 0}^{j_i} \; \left( N_{D/\overline{X}}^\ast \right)^{\otimes (j_1 + ... + j_{i-1} + s_i)} \\ 
	 & \; \; \otimes S^{l_1 - j_1} \Omega_D \otimes ... \otimes S^{l_k - j_k} \Omega_D.
\end{align} 
where all tensor products are taken over $\mathcal O_D$.
\end{prop}

\subsection{Upper bound on the graded terms of the filtration}

We want to obtain an asymptotic upper bound on $h^0 \left(D, \mathrm{Gr}^F_{\bullet} \left( \mathcal Q_{k,m} \right) \right)$ when $m \longrightarrow 0$. We start by changing the indexing of the direct sums, so that we sum over $r = j_1 + 2 j_2 + ... + k j_k$. If we proceed to the substitution $l_i \leftarrow l_i - j_i$, we find:
\begin{align*}
\mathrm{Gr}^F_{\bullet} \left( \mathcal Q_{k,m} \right) = & \bigoplus_{r = 0}^m \left(\, \left[ \bigoplus_{j_1 + 2 j_2 + ... + k j_k = r} \bigoplus_{i = 1}^k \bigoplus_{s_i = 0}^{j_i} \left( N_{D / \overline{X}}^\ast \right)^{\otimes (j_1 + ... + j_{i-1} + s_i)} \right] \right. \\
 & \left. \otimes \left[ \bigoplus_{l_1 + 2 l_2 + ... + k l_k = m - r} S^{l_1} \Omega_D \otimes ... \otimes S^{l_k} \Omega_D \right] \, \right)
\end{align*}

The term on the right is a trivial vector bundle, because $D$ is made of disjoint abelian varieties. Consequently, we have
\begin{align}  \nonumber
h^0 & \left(D, \mathrm{Gr}^F_{\bullet} \left( \mathcal Q_{k,m} \right) \right)  = \\ \nonumber
	&  \sum_{r = 0}^m \; \left( \, \left[ \sum_{j_1 + 2 j_2 + ... + k j_k = r} \sum_{i = 1}^k \sum_{s_i = 0}^{j_i} h^0 \left(D,  \left (N_{D / \overline{X}}^\ast \right)^{\otimes (j_1 + ... + j_{i-1} + s_i)} \right) \right] \right. \\
	& \left. \cdot \mathrm{rk} \left[ \bigoplus_{l_1 + 2 l_2 + ... + k l_k = m - r} S^{l_1} \Omega_D \otimes ... \otimes S^{l_k} \Omega_D \right] \, \right). \label{majorationh0}
\end{align}

For a fixed $(j_1, ..., j_k)$, we now compute 
$$
\sum_{i=1}^k \sum_{s_i = 0}^{j_i} h^0 \left(D,  \left (N_{D / \overline{X}}^\ast \right)^{\otimes (j_1 + ... + j_{i-1} + s_i)} \right). 
$$

Recall that the line bundle $N_{D/ \overline{X}}^\ast$ is ample (cf. \cite{mok12}). Consequently, since the boundary is made of abelian varieties, Kodaira vanishing theorem yields 
$$
\chi(D, (N^\ast_{D / \overline{X}})^{\otimes(j_1 + j_2 + ... + s_i)}) = h^0(D, (N^\ast_{D / \overline{X}})^{\otimes(j_1 + j_2 + ... + s_i)}),
$$
as soon as $j_1 + j_2 + ... + s_i \neq 0$.

Besides, still because the boundary is a union of abelian varieties, Hirzebruch-Riemann-Roch theorem gives 
$$
\chi( D, (N_{D/\overline{X}} ^\ast)^{\otimes (j_1 + j_2 + ... j_{i-1} + s_i)}) = \frac{1}{(n-1)!} (j_1 + ... + j_{i-1} + s_i)^{n-1} \; [ -(-D)^n].
$$

We can sum this last term on $s_i$, to find
\begin{align*}
\sum_{s_i = 0}^{j_i} & (j_1 + ... + j_{i -1}  + s_i )^{n-1} \\
 & = \sum_{l_1 + ... + l_i = n -1} \binom{n-1}{l_1,\; ... \,,\; l_i} j_1^{l_1} ... j_{i-1}^{l_{i-1}} \sum_{s_i = 0}^{j_i} s_i^{l_i} \\
	& =   \sum_{l_1 + ... + l_i = n-1} \binom{n-1}{l_1,\; ... \,,\; l_i} j_1^{l_1} ... j_{i-1}^{l_{i-1}} [\frac{j_i^{l_i + 1}}{l_i + 1} + O(j_i^{l_i -1}) ] \\
	& = \sum_{l_1 + ... + (l_i + 1) = n} \frac{1}{n} \binom{n}{l_1 \; ... \,,  \; l_i + 1} j_1^{l_1} ... j_i^{l_i + 1} + O(  \sum_{l_1 + ... + l_i = n - 1} \binom{n-1}{l_1,\; ... \,,\; l_i} j_1^{l_1} ... j_{i}^{l_{i}} ) \\
	& = \frac{1}{n} \left[ (j_1 + ... + j_i)^n - (j_1 + ... + j_{i-1})^n \right] + O( (j_1 + ... + j_i)^{n-1}),
\end{align*}
where we use the multinomial formula at the second and fourth lines. 

If we sum over $i$, and using the fact that for a fixed $i$, there is only one term for which $j_1 + ... + j_{i-1} + s_i = 0$, we finally find
\begin{align} \label{sommepart}
\sum_{i=1}^k \sum_{s_i = 0}^{j_i} h^0 \left(D,  \left (N_{D / \overline{X}}^\ast \right)^{\otimes (j_1 + ... + j_{i-1} + s_i)} \right) = & \frac{1}{n!} (j_1 + ... + j_n)^n  \frac{[- (-D)^n]}{n!} \\
\nonumber	&  + O(\sum_{i} (j_1 + ... + j_i)^{n-1}).
\end{align}

\subsection{Final asymptotic estimate over $h^0(\mathcal Q_{k,m})$ as $m \longrightarrow + \infty$.}

Applying Proposition \ref{propcombmajorationapplication}, we can sum \eqref{sommepart} over the $j_1, ..., j_k$ such that $j_1 + 2 j_2 +... + k j_k = r$, to find
\begin{align*}
\sum_{j_1 + 2 j_2 + ... + k j_k = r} \sum_{i = 1}^k \sum_{s_i = 0}^{j_i} & h^0 \left(D,  \left (N_{D / \overline{X}}^\ast \right)^{\otimes (j_1 + ... + j_{i-1} + s_i)} \right) \\
	 & \leq \frac{r^{n+k-1}}{(n + k -1)} \left[ \frac{1}{k!} \sum_{1 \leq i_1 \leq ... \leq i_n \leq k} \frac{1}{i_1 ... i_n} \right] [-(-D)^n] + O(r^{n+k-2}). 
\end{align*}
Moreover, according to Proposition \ref{propestimeepoint}, we have 
$$
\mathrm{rk} \left( \bigoplus_{l_1 + 2 l_2 + ... + k l_k = r} S^{l_1} \Omega_D \otimes ... \otimes S^{l_k} \Omega_D \right) = \frac{1}{(k!)^{n-1}} \frac{m^{(n-1)k -1}}{((n-1)k -1)!} + O(m^{nk -2}).
$$

If we put these two asymptotic expressions in \eqref{majorationh0}, we obtain the following final estimate on $h^0(\mathcal Q_{k,m})$, when $m \longrightarrow + \infty$ :
\begin{align} \nonumber 
h^0 & (\mathcal Q_{k,m})  \leq h^0( \mathrm{Gr}_{\bullet}^F (\mathcal Q_{k,m})) \\ \nonumber
	  & \leq \sum_{r = 0}^m \left[\frac{r^{n+k-1}}{(n + k -1)} \left[ \frac{1}{k!} \sum_{1 \leq i_1 \leq ... \leq i_n \leq k} \frac{1}{i_1 ... i_n} \right] [-(-D)^n] \right] \cdot \left[\frac{1}{(k!)^{n-1}} \frac{(m-r)^{(n-1)k -1}}{((n-1)k -1)} \right] \\ \nonumber 
	& \; \; + O( m^{n + nk - 2}) \\  
	& = \frac{[-(-D)^n]}{(k!)^n} \left[\sum_{1 \leq i_1 \leq ... \leq i_n \leq k} \frac{1}{i_1 ... i_k}\right] \frac{m^{n + nk -1}}{(n + nk - 1)!} + O(m^{n + nk - 2}). \label{estimeebord}
\end{align}

\subsection{Uniform lower bound in $k$ on $\mathrm{vol} (E^{GG}_{k,m} \Omega_{\overline{X}})$}

Combining \eqref{minorationstdlog} with \eqref{estimeevolcomb} and \eqref{estimeebord}, we finally obtain the lower bound \eqref{minorationfinale} on $\mathrm{vol}(E_{k, \bullet}^{GG} \Omega_{\overline{X}})$, which proves Theorem \ref{thmvol}.
The expression \eqref{minorationfinale} being valid for any $k$, we can use the results of \cite{baktsi15} to determine an order $k$ after which the algebra  $E^{GG}_{k,\bullet} \Omega_{\overline{X}}$ has maximal growth. 

For example, it is not hard to obtain an asymptotic expansion of \eqref{minorationfinale}, with leading coefficient
$$
\frac{1}{n! (k!)^n} (\log k)^n \left( (K_{\overline{X}} + D)^n + (- D)^n \right) = \frac{1}{n! (k!)^n} (\log k)^k (K_{\overline{X}})^n.
$$

When $K_{\overline{X}}$ is nef and big, we get back the asymptotic lower bound of \cite{dem11}.
$$
\mathrm{vol} (E^{GG}_{k,m} \Omega_{\overline{X}}) \geq \frac{(\log k)^n}{n! (k!)^n} \left( \mathrm{vol} (K_{\overline{X}}) + O((\log k)^{-1})\right).
$$

\subsection{Explicit orders $k$ to have a big $E_{k,\bullet}^{GG} \Omega_{\overline{X}}$}

In this section, we prove Corollary \ref{corolorder}. We will use \eqref{minorationfinale} to determine an effective $k$ after which $E_{k, \bullet}^{GG} \Omega_{\overline{X}}$ is big. Let us begin by determining an upper bound on $\sum_{1 \leq i_1 \leq ... \leq i_n \leq k} \frac{1}{i_1 ... i_n}$. We have
$$
\sum_{1 \leq i_1 \leq ... \leq i_n \leq k} \frac{1}{i_1 ... i_k}  = \sum_{p=1}^n \; \sum_{\overset{l_1 + ... +l_p = n}{\forall i, \; l_i >\, 0}} \; \; \sum_{1 \leq j_1 < ... < j_p \leq k} \; \; \frac{1}{j_1^{l_1} ... j_p^{l_p}},
$$
the datum of $n$ integers $1 \leq i_1 \leq ... \leq i_n \leq k$ in non-decreasing order being equivalent to the one of an integer $p$ giving the number of distinct $i_j$, of  $p$ integers $1 \leq j_1 < ... < j_p \leq k$, and of positive exponents $l_1, ..., l_p$ such that $\sum_{k} l_k = n$. Now, for any $p \geq 1$, we have:
\begin{align*}
\sum_{1 \leq j_1 < ...< j_p \leq n} \frac{1}{j_1 ... j_p}
	& \leq \frac{1}{p!} \sum_{1 \leq j_1, ..., j_p \leq k} \frac{1}{j_1 ... j_p} \\
	& = \frac{1}{p!} \left[ \sum_{j = 1}^{k} \frac{1}{j} \right]^p \\
	& \leq \frac{1}{p!} \left( \log k + \gamma + \frac{1}{2} \right)^p.
\end{align*}
Let $p \leq n-1$, and choose $l_1 ,..., l_p$ such that $l_1 + ... + l_p = n$ and $l_i \neq 0$ for any $i$. Necessarily, at least one $l_i$ is larger than 2, so
\begin{align*}
\sum_{1 \leq j_1 < ... < j_p \leq n} \frac{1}{j_1^{l_1} ... j_p^{l_p}} & \leq \sum_{1 \leq j_1 < ... < j_{p-1} \leq n} \sum_{1 \leq j_p \leq n} \frac{1}{j_1 ... j_{p-1} j_p^2} \\
	& \leq  \sum_{1 \leq j_1 < ...< j_{p-1} \leq n} \frac{1}{j_1 ... j_{p-1}} \cdot \frac{\pi^2}{6} \\
	& \leq \frac{1}{(p-1)!} \left( \log k + \gamma + \frac{1}{2} \right)^{p-1} \cdot \frac{\pi^2}{6}.
\end{align*}
Thus,
\begin{align*}
\sum_{1 \leq i_1 \leq ... \leq i_n \leq k} \frac{1}{i_1 ... i_k}  \leq \frac{1}{n!} & \left( \log k + \gamma + \frac{1}{2} \right)^n  \\
& + \frac{\pi^2}{6} \sum_{p=1}^{n-1} \; \left (\sum_{\overset{l_1 + ... +l_p = n}{\forall i, \; l_i >\, 0}} \; 1 \right) \cdot \frac{1}{(p-1)!} \left( \log k + \gamma + \frac{1}{2} \right)^{p-1}.
\end{align*}

It is easy to see that $\sum_{\overset{l_1 + ... +l_p = n}{\forall i, \; l_i >\, 0}} \; 1 = \binom{n-1}{p-1}$ (choosing the integers $l_i$ amounts to choosing $p-1$ cuts in the set $\left[|1, n \right|]$, i.e. among $n-1$ possible cuts). Consequently, we find
\begin{align*}
\sum_{1 \leq i_1 \leq ...\leq i_n \leq k} \frac{1}{i_1 \; ... \; i_n} \leq 
 	\frac{ \left(\log k + \gamma + \frac{1}{2} \right)^n}{n!} + \frac{\pi^2}{6} \sum_{p=1}^{n-1} \binom{n-1}{p-1} \frac{1}{(p-1)!} \left( \log k + \gamma + \frac{1}{2} \right)^{p-1}.
\end{align*}

We can use the following upper bound:
\begin{align*}
\sum_{p=1}^{n-1} \binom{n-1}{p-1} \frac{1}{(p-1)!} \left( \log k + \gamma + \frac{1}{2} \right)^{p-1} & =  \sum_{p=0}^{n-2} \binom{n-1}{p} \frac{1}{p!} \left( \log k + \gamma + \frac{1}{2} \right)^p \\
	& \leq \sum_{p=0}^{n-2} \binom{n-1}{p} \left( \log k + \gamma + \frac{1}{2} \right)^p \\
	& = \left( \log k + \gamma + \frac{3}{2} \right)^{n-1} - \left( \log k + \gamma + \frac{1}{2} \right)^{n-1} \\
	& \leq (n-2) \left( \log k + \gamma + \frac{3}{2} \right)^{n-2},
\end{align*}
where we used the mean value inequality in the last line.
Thus,
\begin{equation} \label{majorationestimeebord}
\sum_{1 \leq i_1 \leq ...\leq i_n \leq k} \frac{1}{i_1 \; ... \; i_n} \leq \frac{\left(\log k + \gamma + \frac{1}{2} \right)^n}{n!}  + \frac{\pi^2}{6} (n-2) \left( \log k + \gamma + \frac{3}{2} \right)^{n-2}.
\end{equation}

Inserting \eqref{minorationestimeeouvert} and \eqref{majorationestimeebord}, in \eqref{minorationfinale}, we find a lower bound of the form
\begin{align*}
\mathrm{vol} \left( E_{k, \bullet}^{GG} \Omega_{\overline{X}} \right) & \geq C_k \left[ \left(K_{\overline{X}} + D \right)^n + A(k,n) (-D)^n\right]  \\
\end{align*}
for a certain $C_k \in \mathbb R^\ast$, and
$$
A(k, n) = \left[ \frac{\log k +  \gamma + \frac{1}{2} }{\log k + \gamma} \right]^n + (n-2) n! \frac{\pi^2}{6} \frac{ \left( \log k + \gamma + \frac{3}{2} \right)^{n-2} }{\left( \log k + \gamma \right)^n }
$$

Let us first deal with the case where $n \geq 6$. According to \cite{baktsi15}, we have $\left( K_{\overline{X}} + D \right)^n + \alpha (-D)^n > 0$ for all $\alpha \in \left] 0 , \left( \frac{n+1}{2 \pi} \right)^{n} \right[$. The only thing left now is to determine an integer $k$ such that $A(k, n) < \left( \frac{n+1}{2 \pi} \right)^{n}$.

Let $j = \log k + \gamma$. We have
\begin{align*}
A(k, n) & = \left( 1 + \frac{1}{2 j} \right)^n + \frac{\pi^2}{6} \frac{(n-2) n!}{j} \left( 1 + \frac{3}{2 j} \right)^{n-2} \\
	& \leq \left( 1 + \frac{3}{2j} \right)^{n-2} \left( (1 + \frac{1}{2j})^2 + \frac{\pi^2}{6} \frac{(n-2) n!}{j} \right).
\end{align*}

We see that if $j > \frac{\frac{\pi^2}{6} (n-2)n! + 1}{ \frac{n+1}{2 \pi} - 1}$, then $A(k, n) < \left( \frac{n+1}{2 \pi} \right)^n$.

Besides, if $n \in \left[| 4, 5 \right|]$, then $(K_{\overline{X}})^n = \left( K_{\overline{X}} + D \right)^n + (-D)^n > 0$.
Consequently, since $(K_{\overline{X}})^n$ is an integer, $(K_{\overline{X}} + D)^n + (-D)^n \geq 1$, and 
$$
(K_{\overline{X}} + D)^n + \lambda (-D)^n > 0
$$
for any $\lambda \in ]0, 1 + \frac{1}{-(-D)^n} [$.
Thus, $\mathrm{vol} (E_{k, \bullet}^{GG} \Omega_{X}) > 0$ as soon as $A(k, n) < 1 + \frac{1}{(-D)^n}$.
Performing the same computations as before, we see that it is true if
$$
\log k + \gamma > -(-D)^n \left( (n-2) n ! + 1 \right).
$$

We have consequently proved Corollary \ref{corolorder}. 

{
\footnotesize
\bibliographystyle{amsalpha}
\bibliography{biblio}
}

\textsc{Beno\^it~Cadorel, Institut de Math\'ematiques de Toulouse (IMT), UMR 5219, Universit\'e Paul Sabatier, CNRS, 118 route de Narbonne, F-31062 Toulouse Cedex 9, France} \par\nopagebreak
   \textit{E-mail address}: \texttt{benoit.cadorel@math.univ-toulouse.fr}
\vspace{0pt}
\end{document}